\documentclass[11pt,reqno]{amsart}
\usepackage{indentfirst, amssymb,amsmath,amsthm, mathrsfs,cite,graphicx,float}
\newtheorem*{theoA}{Theorem A}
\newtheorem*{theoB}{Theorem B}

\newtheorem{theo}{Theorem}[section]
\newtheorem{lem}{\textrm{Lemma}}[section]

\newtheorem{prob}{Problem}[section]
\newtheorem{open problem}{Open problem}[section]
\newcommand{\pa}{\partial}
\newcommand{\ol}{\overline}
\newcommand{\be}{\begin{equation}}
\newcommand{\ee}{\end{equation}}
\newcommand{\bs}{\begin{small}}
\newcommand{\es}{\end{small}}
\newcommand{\beas}{\begin{eqnarray*}}
\newcommand{\eeas}{\end{eqnarray*}}
\newcommand{\bea}{\begin{eqnarray}}
\newcommand{\eea}{\end{eqnarray}}
\newcommand{\D}{\mathbb{D}}

\renewcommand{\epsilon}{\varepsilon}
\numberwithin{equation}{section}

\begin{document}
\title[Landau-type Theorems]{The Landau-type theorems for functions with logharmonic Laplacian and bounded length distortions}
\author[S. K. Guin and R. Mandal]{Sudip Kumar Guin and Rajib Mandal}
\date{}
\address{Sudip Kumar Guin, Department of Mathematics, Raiganj University, Raiganj, West Bengal-733134, India.}
\email{sudipguin20@gmail.com}
\address{Rajib Mandal, Department of Mathematics, Raiganj University, Raiganj, West Bengal-733134, India.}
\email{rajibmathresearch@gmail.com}
\maketitle
\let\thefootnote\relax
\begin{abstract}
In this study, we establish certain Landau-type theorems for functions with logharmonic Laplacian of the form $F(z)=|z|^2L(z)+K(z)$, $|z|<1$, where $L$ is logharmonic and $K$ is harmonic, with $L$ and $K$ having bounded length distortion in the unit disk $\D=\{z\in\mathbb{C}:|z|<1\}$. Furthermore, we examine the univalence of the mappings $D(F)$, where $D$ is a differential operator.
\end{abstract}
\section{Introduction and Preliminaries}
Let $\D_r=\{z\in\mathbb{C}:|z|<r\}$ be the open disk with center at the origin and radius $r>0$ and $\D:=\D_1$. Let $\mathcal{H}(\D)$ be the class of analytic functions defined on $\D$.
If the complex-valued function $f$ satisfies the Laplace's equation $\Delta f=0$, then $f$ is a harmonic function, where $\Delta $ stands for the Laplacian operator 
\beas \Delta :=4\frac{\pa^2}{\pa z \pa\bar{z}}=\frac{\pa^2}{\pa x^2}+\frac{\pa^2}{\pa y^2}.\eeas
The Jacobian of $f$ is defined as
\beas J_f=|f_z|^2-|f_{\ol{z}}|^2.\eeas
We denote
\beas \Lambda_f(z) &=&\max\limits_{0\leq t \leq 2\pi} \left|f_z(z)+e^{-2it}f_{\ol{z}}(z)\right|=|f_z(z)|+|f_{\ol{z}}(z)|\eeas
and
\beas \lambda_f(z) &=&\min\limits_{0\leq t \leq 2\pi} \left|f_z(z)+e^{-2it}f_{\ol{z}}(z)\right|=\left||f_z(z)|-|f_{\ol{z}}(z)|\right|.\eeas
Note that a harmonic function $f$ is locally univalent if $J_f\ne 0$ (see \cite{CS1984,D2004}). Whenever $\Lambda_f$ is bounded, then it is known that $f$ has bounded length distortion.\\[2mm]
\indent
In a domain $\Omega$, a $2p$-times $(p\geq 1)$ continuously differentiable complex-valued mapping $F$ is called polyharmonic if it satisfies the polyharmonic equation $\Delta^pF=\Delta(\Delta^{p-1}F)=0$. It is clear that the case $p=1$ (resp. $p=2$) reduces to harmonic (resp. biharmonic) mapping. When $\Omega$ is a simply connected domain, then $F$ is polyharmonic in $\Omega$ if, and only, if $F$ has the form 
\beas F(z)=\sum\limits_{k=1}^p|z|^{2(k-1)}G_k(z),\;\;z\in\Omega,\eeas
where $G_k(z)$ is a complex-valued harmonic mapping in $\Omega$ for $k\in\left\{1,2,\dots,p\right\}$ (see \cite{CPW2011,LL2014}).\\[2mm]
\indent The solution of the non-linear elliptic partial differential equation \beas\frac{\ol{f_{\ol{z}}}}{\ol{f}}=a\frac{f_z}{f} ,\eeas is known to be a logharmonic function, where $a\in\mathcal{H}(\D)$ is the second dilatation function with $|a(z)|<1$, $z\in\D$ (see \cite{AB1988,AH1987}). Let $f$ be a univalent logharmonic function with respect to $a$ with $a(0)=0$. If $f(0)=0$, then $f$ is expressed as
\bea\label{eq1.2}f(z)=h(z)\ol{g(z)},\eea where $h(z)=z+\sum_{n=2}^{\infty}a_nz^n $ and $g(z)=1+\sum_{n=1}^{\infty}b_nz^n$. If $0\notin f(\D)$,
then $\log{f(z)}$ is univalent and harmonic, where $f$ has the form (\ref{eq1.2}) with $h$ and $g$ are nonvanishing analytic functions in $\D$. For a detailed study of harmonic functions and logharmonic functions, we refer to \cite{A2009,AAA2012,AB1988,AH1987,BL2019,D2004,L2008,LL2014,LL2019,LL2021,WPL2024} and the references therein.\\[2mm]
\indent In this paper, we deal with the class of all continuous complex-valued functions $F=u+iv$ defined in a domain $\Omega\subseteq\mathbb{C}$ for which the Laplacian of $F$ is logharmonic. Clearly, $\log(\Delta F)$ is harmonic in $\Omega$ if it satisfies the equation $\Delta (\log{(\Delta F)})=0$. If $\Omega$ is any simply connected domain, we can write
\bea\label{eq1.3} F(z)=r^2L(z)+K(z),\quad z=re^{i\theta}\eea
with $L=h_1\ol{g_1}$ is logharmonic and $K=h_2+\ol{g_2}$ is harmonic in $\Omega$,
where $h_1,g_1,h_2$ and $g_2$ are analytic in $\Omega$.
Let $L_{Lh}(\D)$ denote the collection of all functions of the kind (\ref{eq1.3}) defined on $\D$ (see \cite{A2009}).\\[2mm]
\indent Now we recall the classical Landau theorem (see \cite{L1926}): if $f$ is a normalized analytic function in the open unit disk $\D$ with $|f(z)| < M $ for $z\in \D$, then
$f$ is univalent in the open disk $\D_{\rho_0}$ with $\rho_0=1/(M+\sqrt{M^2-1})$, and $f(\D_{\rho_0})$ contains an open disk $\D_{R_0}$ with $R_0=M\rho_0^2$. The result is sharp, with the extremal function $Mz(1-Mz)/(M-z)$.\\
Furthermore, The classical Bloch theorem states that: there exists a positive constant $b$ such that if $f$ is analytic on $\D$ with $f'(0)=1$, then the image $f(\D)$ contains a schlicht disk of radius $b$, {\it i.e.}, a disk of radius $b$, which is the univalent image of some region in $\D$. The Bloch constant is defined as the supremum of such constants $b$ (see \cite{CGH2000,GK2003}).\\[2mm]
\indent Due to Mocanu \cite{M1980}, we consider the differential operator $D$ defined by
\beas D=z\frac{\pa}{\pa z}-\ol{z}\frac{\pa}{\pa \ol{z}}\eeas
on the class of continuously differentiable complex-valued functions. It is easy to check that $D$ is a complex linear operator satisfying:
\beas D(af+bg)=aD(f)+bD(g)\;\;\text{and}\;\; D(fg)=fD(g)+gD(f),\eeas
where $a,b\in\mathbb{C}$ and $f,g$ are continuously differentiable complex-valued functions.\\[2mm]
\indent In 2012, Abdulhadi {\it et al.} \cite{AAA2012} established the following Landau-type theorem for functions with logharmonic Laplacian in $\D$.
\begin{theoA}\cite[Theorem 3]{AAA2012}\label{thA}
Let $F(z)=r^2L(z)+K(z)$ be in $L_{Lh}(\D)$, where $L(z)$ is logharmonic and $K(z)$ is harmonic in the unit disk $\D$ such that
$L(0)=K(0)=0, J_{F}(0)=1$ and $L(z)$, $K(z)$ are both bounded by $M$ in $\D$. Then $F$ is univalent in the disk $|z|<\varrho_1$, where $\varrho_1$ satisfies the equation
\beas \frac{\pi}{4M}-2\varrho M-2M\left(\frac{\varrho^3}{(1-\varrho^2)^2}+\frac{1}{(1-\varrho)^2}-1\right)=0,\eeas
and $F(\D_{\varrho_1})$ contains a disk $\D_{R_1}$ with
\beas R_1=\frac{\pi}{4M}\varrho_1-\varrho_1^2M\frac{1}{1-\varrho_1^2}-2M\frac{\varrho_1^2}{1-\varrho_1}.\eeas
This result is not sharp.
\end{theoA}
In \cite{WPL2024}, Wang {\it et al.} established the following Landau type theorem for the mappings of the form $D(F)$, where $F$ is a polyharmonic mapping with bounded length distortion.
\begin{theoB}\cite[Theorem 3.3]{WPL2024}\label{thB}
Suppose that $F$ is a polyharmonic mapping of the form
\beas F(z)=
a_0+\sum\limits_{k=1}^p|z|^{2(k-1)}\sum\limits_{k=1}^\infty(a_{n,k}z^n+\ol{b_{n,k}z^n}),\eeas
and all its non-zero coefficients $a_{n,k_1}, a_{n,k_2}$ and $b_{n,k_3},b_{n,k_4}$ satisfy the condition:
\beas \left|\arg{\frac{a_{n,k_1}}{a_{n,k_2}}}\right|\leq \frac{\pi}{2},\;\;\text{and}\;\;\left|\arg{\frac{b_{n,k_3}}{b_{n,k_4}}}\right|\leq \frac{\pi}{2}.\eeas
If $\Lambda_F(z)\leq M$ in $\D$ for some $M>1$ and $F(0)=J_F(0)-1=0$, then $D(F)$ is univalent in $\D_{\varrho_2}$, where $\varrho_2$ is the least positive root of the equation
\beas &&1-\sqrt{M^4-1}\left(\frac{2r-r^2}{(1-r)^2}+\sum\limits_{k=2}^p \frac{2r^{2k-1}}{\sqrt{10}(1-r)^3}\right.\\[2mm] &&
\left.+\sum\limits_{k=2}^p(2k-1)r^{2(k-1)}\left(\frac{1}{\sqrt{5}}+\frac{2r-r^2}{\sqrt{10}(1-r)^2}\right)\right)=0,\eeas
and $F(\D_{\varrho_2})$ contains a univalent disk $\D_{R_2}$ with
\beas R_2=\lambda_0(M)\varrho_2\left(1-\sqrt{M^4-1}\frac{\varrho_2}{1-\varrho_2}-\sqrt{M^4-1}\sum\limits_{k=2}^p\varrho_2^{2(k-1)}\left(\frac{1}{\sqrt{5}}+\frac{2\varrho_2-\varrho_2^2}{\sqrt{10}(1-\varrho_2)^2}\right)\right),\eeas
where
\beas \lambda_0(M)=\frac{\sqrt{2}}{\sqrt{M^2-1}+\sqrt{M^2+1}}.\eeas
In the case of $M=1$, $D(F)(z)=\alpha z$ with $|\alpha|=1$.
\end{theoB}
Now the results due to Abdulhadi {\it et al.} \cite{AAA2012} and Wang {\it et al.} \cite{WPL2024} motivate us to raise the following problem:
\begin{prob}\label{prob1}
Can we establish the Landau-type theorems for functions in the class $L_{Lh}(\D)$ with bounded length distortions?
\end{prob}
The objective of this paper is to find an affirmative answer to Problem \ref{prob1}.\\[2mm]
This paper is organized into the following sections. In section 2, we first establish the coefficient estimates for logharmonic mapping with bounded length distortions in the unit disk. In section 3, we present four certain Landau-type theorems for functions in the class $L_{Lh}(\D)$ with bounded length distortions.
\section{Key Lemmas}
The subsequent lemmas are pivotal in substantiating our principal findings.
\begin{lem}\cite{L2008}\label{l4}
Let $f(z)=h(z)+\ol{g(z)}$ be a harmonic mapping on $\D$ such that $h(z)=\sum_{n=1}^{\infty}a_nz^n,g(z)=\sum_{n=1}^{\infty}b_nz^n$ with $\lambda_f(0)=1$
and $\Lambda_{f}\leq \Lambda$. Then $\Lambda\geq 1$ and
\beas |a_n|+|b_n|\leq \frac{\Lambda^2 -1}{n \Lambda}\quad\text{for}\quad n\geq 2.\eeas
When $\Lambda >1$, the above estimate is sharp and the extremal functions $f_n(z)$ and $\ol{f_n(z)}$ are given by
\beas f_n(z)=\Lambda^2 z-(\Lambda^3-\Lambda)\int_0^z\frac{dz}{\Lambda +z^{n-1}}\quad\text{for}\quad n\geq 2.\eeas
When $\Lambda=1$, then $f(z)=a_1z+\ol{b_1z}$ with $||a_1|-|b_1||=1$.
\end{lem}
\begin{lem}\cite{LL2019}\label{l5}
Let $H(z)$ be a harmonic mapping on $\D$ with $\lambda_H(0)=1$ and $\Lambda_H(z)<\Lambda\;(\Lambda > 1)$ for all $z\in \D$. Then for all $z_1,z_2\in\D_r\;(0<r<1,\;z_1\neq z_2)$, we have
\beas \left|\int_{[z_1,z_2]}H_z(z)dz+H_{\ol{z}}d\ol{z}\right|\geq \Lambda \frac{1-\Lambda r}{\Lambda-r}|z_2-z_1|,\eeas
where $[z_1,z_2]$ represents the line segment joining $z_1$ and $z_2$.
\end{lem}
\begin{lem}\cite{LL2019}\label{l8}
Let $H(z)$ be a harmonic mapping on $\D$ with $\lambda_H(0)=1$ and $\Lambda_H(z)<\Lambda\;(\Lambda > 1)$ for all $z\in \D$.
Set $\gamma =H^{-1}(\ol{ow'})$ with $w'\in H(\pa \D_r)\;(0<r\le 1)$, where $\ol{ow'}$ denotes the closed line segment joining the origin and $w'$. Then
\beas \left|\int_{\gamma}H_\zeta (\zeta )d\zeta+H_{\ol{\zeta}}d\ol{\zeta}\right|\geq
\Lambda \int_0^{r}\frac{\frac{1}{\Lambda}-t}{1-\frac{t}{\Lambda}}dt.\eeas
\end{lem}
For the coefficient estimates of logharmonic mapping with bounded length distortions in the unit disk, we establish the following two lemmas.
\begin{lem}\label{l12}
Let $L(z)=h(z)\ol{g(z)}$ be logharmonic in $\D$ with $h(z)=z+\sum_{n=2}^\infty a_nz^n$ and $g(z)=1+\sum_{n=1}^\infty b_nz^n$ such that $\Lambda_{L_z}\leq \Lambda_1 $ for all $z\in \D$. Then for $n\geq 0$, we have
\beas \left|\sum_{j=0}^n (j+1)(j+2)a_{j+2}b_{n-j}\right|\leq \Lambda_1 \;\;\text{and}\;\; \left|\sum_{j=0}^n (j+1)(n-j+1)a_{j+1}b_{n-j+1}\right|\leq \Lambda_1.\eeas
\end{lem}
\begin{proof}
The product of two power series, $r(z)=\sum_{n=0}^\infty c_nz^n$ and $s(z)=\sum_{n=0}^\infty d_nz^n$, is given by
\beas t(z)=\left(\sum_{n=0}^\infty c_nz^n\right)\left( \sum_{n=0}^\infty d_nz^n\right)= \sum_{n=0}^\infty \left(\sum_{j=0}^n c_jd_{n-j}\right)z^n.\eeas
Set $a_1=b_0=1$. For $L(z)=h(z)\ol{g(z)}$, we have
\beas L_{zz}(z)&=& h''(z)\ol{g(z)}=\left(\sum_{n=2}^\infty n(n-1)a_nz^{n-2}\right)\left(\ol{\sum_{n=0}^\infty b_nz^n}\right)\\[2mm]
&=& \left(\sum_{n=0}^\infty (n+1)(n+2)a_{n+2}z^{n}\right)\left(\ol{\sum_{n=0}^\infty b_nz^n}\right) \eeas
\beas &= & \sum_{n=0}^\infty \sum_{j=0}^n (j+1)(j+2)a_{j+2}\ol{b_{n-j}}z^{j}\ol{z}^{n-j}\quad
\text{and}\eeas
\beas L_{\ol{z}z}(z)&=& h'(z)\ol{g'(z)}=\left(\sum_{n=1}^\infty na_nz^{n-1}\right)\left(\ol{\sum_{n=1}^\infty nb_nz^{n-1}}\right)\\[2mm]
&=& \left(\sum_{n=0}^\infty (n+1)a_{n+1}z^{n}\right)\left(\ol{\sum_{n=0}^\infty (n+1)b_{n+1} z^{n}}\right)\\[2mm]
&=& \sum_{n=0}^\infty \sum_{j=0}^n (j+1)(n-j+1)a_{j+1}\ol{b_{n-j+1}} z^{j}\ol{z}^{n-j}.
\eeas
Then Parseval's identity follows that
\bea\label{p1} \frac{1}{2\pi}\int_0^{2\pi}|L_{zz}(re^{i\theta})|^2d\theta &=& \sum_{n=0}^\infty \left|\sum_{j=0}^n (j+1)(j+2)a_{j+2}b_{n-j}\right|^2r^{2n}\eea
and
\bea\label{p2}\qquad \frac{1}{2\pi}\int_0^{2\pi}|L_{\ol{z}z}(re^{i\theta})|^2d\theta &=&  \sum_{n=0}^\infty \left|\sum_{j=0}^n (j+1)(n-j+1)a_{j+1}b_{n-j+1}\right|^2r^{2n}.\eea
Also, since $\Lambda_{L_z}\leq \Lambda_1$ for all $z\in\D$, we have
\bea\label{p3} \frac{1}{2\pi}\int_0^{2\pi}\left(|L_{zz}(re^{i\theta})|^2+|L_{\ol{z}z}(re^{i\theta})|^2\right)d\theta
&\leq &\frac{1}{2\pi}\int_0^{2\pi}\left(|L_{zz}(re^{i\theta})|+|L_{\ol{z}z}(re^{i\theta})|\right)^2d\theta\nonumber\\ &\leq &{\Lambda_1}^2.\eea
Hence, using (\ref{p1}) and (\ref{p2}) in (\ref{p3}), we have
\beas \sum_{n=0}^\infty \left|\sum_{j=0}^n (j+1)(j+2)a_{j+2}b_{n-j}\right|^2r^{2n}+\sum_{n=0}^\infty \left|\sum_{j=0}^n (j+1)(n-j+1)a_{j+1}b_{n-j+1}\right|^2r^{2n}\leq {\Lambda_1}^2.\eeas
Note that $r\to 1^{-}$ in the above inequality yields
\beas \sum_{n=0}^\infty \left|\sum_{j=0}^n (j+1)(j+2)a_{j+2}b_{n-j}\right|^2+\sum_{n=0}^\infty \left|\sum_{j=0}^n (j+1)(n-j+1)a_{j+1}b_{n-j+1}\right|^2\leq {\Lambda_1}^2.\eeas
Therefore for $n=0,1,2,\dots$, we have
\beas \left|\sum_{j=0}^n (j+1)(j+2)a_{j+2}b_{n-j}\right|\leq \Lambda_1 \quad \text{and}\quad \left|\sum_{j=0}^n (j+1)(n-j+1)a_{j+1}b_{n-j+1}\right|\leq \Lambda_1.\eeas
This completes the proof.
\end{proof}
\begin{lem}\label{l13}
Let $L(z)=h(z)\ol{g(z)}$ be logharmonic in $\D$ with $h(z)=z+\sum_{n=2}^\infty a_nz^n$ and $g(z)=1+\sum_{n=1}^\infty b_nz^n$ such that $\Lambda_{L_{\ol{z}}}\leq \Lambda_2 $ for all $z\in \D$. Then for $n\geq 0$, we have
\beas \left|\sum_{j=0}^n (n-j+1)(n-j+2)a_jb_{n-j+2}\right|\leq \Lambda_2 ,\;\left|\sum_{j=0}^n (j+1)(n-j+1)a_{j+1}b_{n-j+1}\right|\leq \Lambda_2.\eeas
\end{lem}
\begin{proof}
Let $a_0=0$ and $a_1=1$. Using arguments similar to those used in the proof of \textrm{Lemma} $\ref{l12}$, we have
\beas L_{\bar{z}\bar{z}}(z)&=& h(z)\ol{g''(z)}=\left(z+\sum_{n=2}^\infty a_nz^n\right)\left(\ol{\sum_{n=2}^\infty n(n-1)b_nz^{n-2}}\right)\\[2mm]
&=& \left(\sum_{n=0}^\infty a_nz^n\right)\left(\ol{\sum_{n=0}^\infty (n+1)(n+2)b_{n+2}z^{n}}\right)\\[2mm]
&=& \sum_{n=0}^\infty \sum_{j=0}^n  (n-j+1)(n-j+2)a_j\ol{b_{n-j+2}}z^j\ol{z}^{n-j}\quad
\text{and}\\
 L_{z\ol{z}}(z)&=& h'(z)\ol{g'(z)}=\left(1+\sum_{n=2}^\infty na_nz^{n-1}\right)\left(\ol{\sum_{n=1}^\infty nb_nz^{n-1}}\right)\\[2mm]
&=& \left(\sum_{n=0}^\infty (n+1)a_{n+1}z^{n}\right)\left(\ol{\sum_{n=0}^\infty (n+1)b_{n+1} z^{n}}\right)\\[2mm]
&=& \sum_{n=0}^\infty \sum_{j=0}^n (j+1)(n-j+1)a_{j+1}\ol{b_{n-j+1}} z^{j}\ol{z}^{n-j}.
\eeas
Thus, for $n\geq 0$, we have
\beas \left|\sum_{j=0}^n (n-j+1)(n-j+2)a_jb_{n-j+2}\right|\leq \Lambda_2,\; \left|\sum_{j=0}^n (j+1)(n-j+1)a_{j+1}b_{n-j+1}\right|\leq \Lambda_2.\eeas
This completes the proof.
\end{proof}
\section{The Landau-type theorems}
In the following two results, we establish the sharp Landau-type theorems for functions in the class $L_{Lh}(\D)$ with bounded length distortions.
\begin{theo}\label{th2}
Let $F(z)=r^2L(z)+K(z)$ be in $L_{Lh}(\D)$, where $L(z)$ is logharmonic and $K(z)$ is harmonic in $\D$ such that $F(0)=\lambda_F(0)-1=0$
with $\Lambda_K < \Lambda_1 \;(\Lambda_1>1)$ and $\Lambda_L \leq \Lambda_2 $. Then there exists a constant $0 < r_1 < 1 $ such that $F$ is univalent in the disk $|z|<r_1$. Specifically, $r_1$ satisfies
\beas \Lambda_1\frac{1-\Lambda_1 r}{\Lambda_1-r}-3r^2\Lambda_2=0,\eeas
and $F(\D_{r_1})$ contains a schlicht disk $\D_{\rho_1}$, with
\beas \rho_1=\Lambda_1^2 r_1+(\Lambda_1^3-\Lambda_1)\ln\left(1-\frac{r_1}{\Lambda_1}\right)-r_1^3\Lambda_2.\eeas
Moreover, these estimates are sharp, with an extremal function given by
\bea\label{ex1} F_0(z)&=&\Lambda_1\int_{[0,z]}\frac{\frac{1}{\Lambda_1}- z}{1-\frac{z}{\Lambda_1}}dz-\Lambda_2 z^2\ol{z}\nonumber\\[2mm]
&=& \Lambda_1^2z+(\Lambda_1^3-\Lambda_1)\ln \left(1-\frac{z}{\Lambda_1}\right)-\Lambda_2 z^2\ol{z}, \;\;z\in\D.\eea
\end{theo}
\begin{proof}
Let $F(z)=r^2L(z)+K(z)$. Then
\beas F_z(z)&=&\ol{z}L(z)+|z|^2L_z(z)+K_z(z)\eeas
and
\beas F_{\ol{z}}(z)&=& zL(z)+|z|^2L_{\ol{z}}(z)+K_{\ol{z}}(z).\eeas
To prove the univalence of $F$ in $\D_{r_1}$, let us assume $z_1,z_2\in\D_r$ with $z_1\neq z_2$ and $r\in (0,1)$. Then for $[z_1,z_2]$, the line segment joining $z_1$ to $z_2$, we have
\bea\label{i4} &&\left|F(z_2)-F(z_1)\right| \nonumber\\[2mm]
&=& \left|\int_{[z_1,z_2]}F_z(z)dz+F_{\ol{z}}(z)d\ol{z}\right| \nonumber\\[2mm]
&=&\left|\int_{[z_1,z_2]}(\ol{z}L(z)+|z|^2L_z(z)+K_z(z))dz+(zL(z)+|z|^2L_{\ol{z}}(z)+K_{\ol{z}}(z))d\ol{z}\right|\nonumber\\[2mm]
&\geq & I_1-I_2-I_3,\eea
where
\beas &&I_1=\left|\int_{[z_1,z_2]} K_z(z)dz+K_{\ol{z}}(z)d{\ol{z}}\right|,\;
I_2=\left|\int_{[z_1,z_2]} |z|^2(L_z(z)dz+L_{\ol{z}}(z)d\ol{z})\right|\\[2mm]
 \text{and}&&I_3= \left|\int_{[z_1,z_2]} L(z)(\ol{z}dz+zd\ol{z})\right|.\eeas
Since $\lambda_F(0)=\lambda_K(0)=1$, and in view of \textrm{Lemma} $\ref{l5}$, we have
\bea\label{a1} I_1=\left|\int_{[z_1,z_2]} K_z(z)dz+K_{\ol{z}}(z)d{\ol{z}}\right|\geq \Lambda_1\frac{1-\Lambda_1 r}{\Lambda_1-r}|z_2-z_1|.\eea
Further, we have
\bea\label{F01} I_2=\left|\int_{[z_1,z_2]} |z|^2(L_z(z)dz+L_{\ol{z}}(z)d\ol{z})\right|\leq r^2\Lambda_2|z_2-z_1|.\eea
For $z\in\D$, we have
\bea\label{eq4.3.1} |L(z)|=\left|\int_{[0,z]}{L_z(z)dz+L_{\ol{z}}(z)d\ol{z}}\right|\leq \int_{[0,z]}{|L_z(z)
+L_{\ol{z}}(z)|}|dz|\leq \Lambda_2|z|.\eea
Using (\ref{eq4.3.1}), we have
\bea\label{F02} I_3 &=& \left|\int_{[z_1,z_2]} L(z)(\ol{z}dz+zd\ol{z})\right|\leq 2r^2\Lambda_2|z_2-z_1|.\eea
In view of (\ref{i4}), (\ref{a1}), (\ref{F01}) and (\ref{F02}), we have
$\left|F(z_2)-F(z_1)\right| \geq |z_2-z_1|g_1(r)$,
where
\beas g_1(r)=\Lambda_1\frac{1-\Lambda_1 r}{\Lambda_1-r}-3r^2\Lambda_2.\eeas
It is evident that
\beas g_1'(r)=\Lambda_1\frac{1-\Lambda_1^2}{(\Lambda_1-r)^2}-6r\Lambda_2 < 0\;\;\text{for}\;\; r\in(0,1).\eeas
Hence, $g_1(r)$ is a strictly decreasing function of $r$ on $[0,1]$. Note that $g_1(0)=1>0$ and $g_1(1/\Lambda_1)=-3\Lambda_2/\Lambda_1^2\leq 0$. Then there exists a unique $r_1\in (0,1/\Lambda_1]$ such that $g_1(r_1)=0$. Thus, we have $|F(z_2)-F(z_1)|>0$, {\it i.e.}, $F(z_1)\neq F(z_2)$ in $\D_{r_1}$, and therefore $F(z)$ is univalent in $\D_{r_1}$.\\[2mm]
\indent Now we prove that $F(\D_{r_1})$ contains a schlicht disk $\D_{\rho_1}$. Note that $F(0)=0$, and for $z'\in\pa\D_{r_1}$ with $w'=F(z')\in F(\pa\D_{r_1})$ and $|w'|=\min{\{|w|:w\in F(\pa\D_{r_1})\}}$, let $\gamma=F^{-1}(\ol{ow'})$. In view of \textrm{Lemma} $\ref{l8}$ and using the inequality $(\ref{eq4.3.1})$, we have
\beas |F(z')|=||z'|^2L(z')+K(z')|&\geq& |K(z')|-r_1^2|L(z')|\\[2mm]
&= & \left|\int_{\gamma}K_{\zeta}(\zeta)d\zeta+K_{\ol{\zeta}}(\zeta)d\ol{\zeta}\right|-r_1^2|L(z')|\\[2mm]
&\geq & \Lambda_1 \int_0^{r_1}\frac{\frac{1}{\Lambda_1}-t}{1-\frac{t}{\Lambda_1}}dt-r_1^3\Lambda_2\\[2mm]
&=& \Lambda_1^2 r_1+(\Lambda_1^3-\Lambda_1)\ln\left(1-\frac{r_1}{\Lambda_1}\right)-r_1^3\Lambda_2=\rho_1.\eeas
\indent Now to prove the sharpness of $r_1$ and $\rho_1$, we consider the function $F_0(z)$ given by (\ref{ex1}) with $L(z)=\Lambda_2 z$ and $K(z)=\Lambda_1^2z+(\Lambda_1^3-\Lambda_1)\ln \left(1-\frac{z}{\Lambda_1}\right)$. It is easy to verify that $F_0(z)$ satisfies all the conditions of \textrm{Theorem} \ref{th2}. Therefore $F_0(z)$ is univalent in $\D_{r_1}$ and $F_0(\D_{r_1})\supseteq \D_{\rho_1}$.\\[2mm]
\indent To show that the radius $r_1$ is sharp, we need to prove that $F_0(z)$ is not univalent in $\D_r$ for each $r\in(r_1,1]$. In fact consider the real differential function
\beas h_0(x)=\Lambda_1^2x+(\Lambda_1^3-\Lambda_1)\ln \left(1-\frac{x}{\Lambda_1}\right)-\Lambda_2 x^3,\;\;x\in[0,1].\eeas
Note that the continuous function 
\beas h_0'(x)=\Lambda_1\frac{1-\Lambda_1 x}{\Lambda_1-x}-3x^2\Lambda_2=g_0(x)\eeas
is strictly decreasing function on $[0,1]$ and $g_0(r_1)=0$. Hence $h_0(x)$ is strictly increasing in $[0,r_1)$ and strictly decreasing in $(r_1,1]$. Then there exists $z_1=x_1$ and $z_2=x_2$, where $z_1,z_2\in\D_r\;(r\in(r_1,1])$ with $z_1\neq z_2$ such that
\beas F_0(z_1)=F_0(x_1)=h_0(x_1)=h_0(x_2)=F_0(x_2)=F_0(z_2) ,\eeas
which shows that $F_0(z)$ is not univalent in $\D_r$ for each $r\in(r_1,1]$. Thus, the univalence radius $r_1$ is sharp.\\[2mm]
\indent Finally, we show the sharpness of $\rho_1$. Note that $F_0(0)=0$, and for $z=r_1\in\partial \D_{r_1}$, we have
$|F_0(z)-F_0(0)|=|F_0(r_1)|=F_0(r_1)=\rho_1$.
This completes the proof.
\end{proof}
In \textrm{Lemma} $\ref{l4}$, for a harmonic mapping $K(z)$ with $K(0)=0,\lambda_K(0)=1$ and $\Lambda_K(z)\leq \Lambda_1$, we have $\Lambda_1\geq 1$. Note that \textrm{Theorem} \ref{th2} deals with the case $\Lambda_1>1$. Therefore, we deal with the case $\Lambda_1=1$ in the following result.
\begin{theo}\label{th21}
Let $F(z)=r^2L(z)+K(z)$ be in $L_{Lh}(\D)$, where $L(z)$ is logharmonic and $K(z)$ is harmonic in $\D$ such that $F(0)=\lambda_F(0)-1=0$
with $\Lambda_K\leq 1$ and $\Lambda_L \leq \Lambda_2 $. Then there exists a constant $0 < r_1' < 1 $ such that $F$ is univalent in the disk $|z|<r_1'$, where
\bea\label{r*} r_1'=\left\{\begin{array}{lll}1, &\text{if}\;\;3\Lambda_2\leq 1, \\\dfrac{1}{\sqrt{3\Lambda_2}}, &\text{if}\;\;3\Lambda_2> 1\end{array}\right.\eea
and $F(\D_{r_1'})$ contains a schlicht disk $\D_{\rho_1'}$, with
$ \rho_1'=r_1'-r_1'^3\Lambda_2$.
Moreover, these estimates are sharp with an extremal function given by
\bea\label{ex2} F_1(z)&=& z-\Lambda_2 z^2\ol{z}, \;\;z\in\D.\eea
\end{theo}
\begin{proof}
In view of \textrm{Lemma} $\ref{l4}$, we have
$ K(z)=a_1z+\ol{b_1z}$
with $||a_1|-|b_1||=1$. Using arguments similar to those used in the proof of \textrm{Theorem} \ref{th2}, we obtain (\ref{i4}), where
\bea\label{F03} I_1=\left|\int_{[z_1,z_2]} K_z(z)dz+K_{\ol{z}}(z)d{\ol{z}}\right|\geq ||a_1|-|b_1|||z_2-z_1|=|z_2-z_1|.\eea
In view of (\ref{i4}), (\ref{F01}), (\ref{F02}) and (\ref{F03}), we have
$ \left|F(z_2)-F(z_1)\right| \geq |z_2-z_1|g_2(r)$,
where $g_2(r)=1-3r^2\Lambda_2$ is a strictly decreasing function of $r$ on $[0,1]$. To prove the univalence of $F$ on the disk $\D_{r_1'}$, where $r_1'$ is given by (\ref{r*}), we consider the following cases.\\
\textbf{Case 1.} If $3\Lambda_2\leq 1$, then it is evident that
$ \left|F(z_2)-F(z_1)\right|>0$ for $r\in(0,1)$. Therefore, $F$ is univalent on $\D$.\\
{\bf Case 2.} If $3\Lambda_2> 1$, then there exists a unique $r_*\in(0,1)$ such that $g_2(r_*)=0$. Hence, $F(z)$ is univalent in $\D_{r_*}$.\\[2mm]
\indent Now we prove that $F(\D_{r_1'})$ contains a schlicht disk $\D_{\rho_1'}$. Note that $F(0)=0$ and in view of $(\ref{eq4.3.1})$, for $z=r_1'e^{i\theta}\in\pa\D_{r_1'}$, we have
\beas |F(z)|=||z|^2L(z)+K(z)|
&\geq &|K(z)|-r_1'^2|L(z)|\\[2mm]
&\geq & r_1'||a_1|-|b_1||-r_1'^3\Lambda_2\\[2mm]
&=&r_1'-r_1'^3\Lambda_2=\rho_1'.\eeas
\indent A similar approach as \textrm{Theorem} \ref{th2} shows the sharpness of $r_1'$ and $\rho_1'$ with an extremal function given by (\ref{ex2}). This completes the proof.
\end{proof} 
In the next two results, we establish the Landau-type theorems for functions of the form $D(F)$, where $F$ belonging to the class $L_{Lh}(\D)$ with bounded length distortions.
\begin{theo}\label{th4}
Let $F(z)=r^2L(z)+K(z)$ be in $L_{Lh}(\D)$, where $L(z)$ is logharmonic and $K(z)$ is harmonic in $\D$ such that $F(0)=\lambda_F(0)-1=0$ with $\Lambda_K\leq \Lambda_1$ $(\Lambda_1>1)$, $ \Lambda_{L_z} \leq \Lambda_2 $ and $\Lambda_{L_{\ol{z}}} \leq \Lambda_3$. Then $D(F)$ is univalent in $|z|<r_2$, where $r_2\in (0,1)$ satisfies the equation
\beas &&1-3(\Lambda_2+\Lambda_3 )r^3 -2(\Lambda_2+\Lambda_3)\frac{r^3}{1-r}-  \frac{\Lambda_1^2-1}{\Lambda_1}\frac{r(2-r)}{(1-r)^2}=0.\eeas
Also, $D(F)$ contains a schlicht disk $\D_{\rho_2}$ with
\beas \rho_2= r_2- \frac{\Lambda_1^2-1}{\Lambda_1}\frac{r_2^2}{1-r_2}-(\Lambda_2+\Lambda_3)r_2^4. \eeas
\end{theo}
\begin{proof}
Let $L(z)$ and $K(z)$ be of the form
\beas L(z)=\left(z+\sum_{n=2}^\infty a_nz^n\right)\left(\ol{1+\sum_{n=1}^\infty b_nz^n}\right) \;\;\text{and}\;\;K(z)=\sum_{n=1}^\infty c_nz^n +\ol{\sum_{n=1}^\infty d_nz^n}.\eeas
Using the differential operator $D$, we have
\beas H(z)&=&D(F)(z)=zF_z(z)-\ol{z}F_{\ol{z}}(z)\\[2mm]
&=& z\left\{\ol{z}L(z)+|z|^2L_z(z)+K_z(z)\right\}-\ol{z}\left\{zL(z)+|z|^2L_{\ol{z}}(z)+K_{\ol{z}}(z)\right\}\\[2mm]
&=& z|z|^2L_z(z)-\ol{z}|z|^2L_{\ol{z}}(z)+zK_z(z)-\ol{z}K_{\ol{z}}(z).\eeas
A simple calculation yields that
\beas H_z(z)=2z\ol{z}L_z(z)+z^2\ol{z}L_{zz}(z)-{\ol{z}}^2L_{\ol{z}}(z)-z{\ol{z}}^2L_{z\ol{z}}(z)+K_z(z)+zK_{zz}(z)-\ol{z}K_{z\ol{z}}(z)\eeas
and
\beas H_{\ol{z}}(z)=z^2L_z(z)+z^2\ol{z}L_{\ol{z}z}(z)-2\ol{z}zL_{\ol{z}}(z)-z{\ol{z}}^2L_{{\bar{z}}{\bar{z}}}(z)+zK_{\ol{z}z}(z)-K_{\ol{z}}(z)
-\ol{z}K_{\bar{z}\bar{z}}(z).\eeas
Therefore,
\bea\label{H01} H_{z}(z)dz+H_{\ol{z}}(z)d\ol{z}
&=& K_z(0)dz-K_{\ol{z}}(0)d{\ol{z}}+2z\ol{z}L_z(z)dz-2\ol{z}zL_{\ol{z}}(z)d{\ol{z}}-{\ol{z}}^2L_{\ol{z}}(z)dz\nonumber\\[2mm]
&&+z^2L_z(z)d{\ol{z}}+z^2\ol{z}L_{zz}(z)dz+z^2\ol{z}L_{\ol{z}z}(z)d{\ol{z}}-z{\ol{z}}^2L_{z\ol{z}}(z)dz\nonumber\\[2mm]
&&-z{\ol{z}}^2L_{{\bar{z}}{\bar{z}}}(z)d{\ol{z}}+(K_z(z)-K_z(0))dz-(K_{\ol{z}}(z)-K_{\ol{z}}(0))d{\ol{z}}\nonumber\\[2mm]
&&+zK_{zz}(z)dz+zK_{\ol{z}z}(z)d{\ol{z}}-\ol{z}K_{z\ol{z}}(z)dz-\ol{z}K_{\bar{z}\bar{z}}(z)d{\ol{z}}.\eea
Since $K(z)$ is harmonic, from (\ref{H01}) it follows that
\bea\label{i7} |H(z_2)-H(z_1)|=\left|\int_{[z_1,z_2]}H_{z}(z)dz+H_{\ol{z}}(z)d\ol{z} \right|\geq  I_1-I_2-I_3-I_4-I_5-I_6,\eea
where
\beas &&I_1 = \left|\int_{[z_1,z_2]} K_z(0)dz-K_{\ol{z}}(0)d\ol{z} \right|,\;
I_2=\left|\int_{[z_1,z_2]}2z\ol{z}L_z(z)dz-2\ol{z}zL_{\ol{z}}(z)d{\ol{z}}\right|,\\[2mm]
&&I_3 = \left|\int_{[z_1,z_2]} -{\ol{z}}^2L_{\ol{z}}(z)dz+z^2L_z(z)d{\ol{z}}\right|,\\[2mm]
&&I_4= \left|\int_{[z_1,z_2]} z^2\ol{z}L_{zz}(z)dz+z^2\ol{z}L_{\ol{z}z}(z)d{\ol{z}} -z{\ol{z}}^2L_{z\ol{z}}(z)dz
-z{\ol{z}}^2L_{{\bar{z}}{\bar{z}}}(z)d{\ol{z}}\right|,\\[2mm]
&&I_5 = \left|\int_{[z_1,z_2]} (K_z(z)-K_z(0))dz-(K_{\ol{z}}(z)-K_{\ol{z}}(0))d\ol{z}\right|\\\text{and}
&&I_6 =  \left|\int_{[z_1,z_2]} zK_{zz}(z)dz-\ol{z}K_{\bar{z}\bar{z}}(z)d{\ol{z}}\right|.\eeas
Since $\lambda_K(0)=\lambda_F(0)=1$, we have
\beas I_1 &=& \left|\int_{[z_1,z_2]} K_z(0)dz-K_{\ol{z}}(0)d\ol{z} \right|\geq \int_{[z_1,z_2]} \lambda_K(0)|dz|=\lambda_K(0)|z_2-z_1|=|z_2-z_1|.\eeas
Now
\bea\label{L10} |L_z(z)|&=&\left|\int_{[0,z]}{L_{zz}(z)dz+L_{\ol{z}z}(z)d\ol{z}}\right|\leq \int_{[0,z]}{(|L_{zz}(z)|+|L_{\ol{z}z}(z)|)}|dz|\leq \Lambda_2|z|\nonumber\\\eea
and
\bea\label{L11} |L_{\ol{z}}(z)|&=&\left|\int_{[0,z]}{L_{z{\ol{z}}}(z)dz+L_{\bar{z}\bar{z}}(z)d\ol{z}}\right|\leq \int_{[0,z]}{(|L_{z{\ol{z}}}(z)|+|L_{\bar{z}\bar{z}}(z)|)}|dz|\leq \Lambda_3|z|.\nonumber\\\eea
In view of (\ref{L10}) and (\ref{L11}), we have
\bea\label{R1} I_2=\left|\int_{[z_1,z_2]}2z\ol{z}L_z(z)dz-2\ol{z}zL_{\ol{z}}(z)d{\ol{z}}\right|
&\leq & 2r^2 \int_{[z_1,z_2]}\left(|L_z(z)|+ |L_{\ol{z}}(z)|\right)|dz|\nonumber\\[2mm]
&\leq & 2r^3 |z_2-z_1|(\Lambda_2+\Lambda_3 ),\\
\label{R2} I_3 = \left|\int_{[z_1,z_2]} -{\ol{z}}^2L_{\ol{z}}(z)dz+z^2L_z(z)d{\ol{z}}\right|
&\leq & r^2\int_{[z_1,z_2]} \left(|L_{\ol{z}}(z)|+|L_{z}(z)|\right)|dz|\nonumber\\[2mm]
& \leq & r^3|z_2-z_1| (\Lambda_2 +\Lambda_3).\eea
In view of \textrm{Lemmas} $\ref{l12}$ and $\ref{l13}$, we have
\bea\label{R3} I_4&=& \left|\int_{[z_1,z_2]} z^2\ol{z}L_{zz}(z)dz+z^2\ol{z}L_{\ol{z}z}(z)d{\ol{z}} -z{\ol{z}}^2L_{z\ol{z}}(z)dz
-z{\ol{z}}^2L_{{\bar{z}}{\bar{z}}}(z)d{\ol{z}}\right|\nonumber\\[2mm]
&\leq & r^3\int_{[z_1,z_2]} \left|L_{zz}(z)+L_{z\ol{z}}(z)+L_{\ol{z}z}(z)+L_{{\bar{z}}{\bar{z}}}(z)\right||dz|\nonumber\\[2mm]
&\leq & r^3  \int_{[z_1,z_2]}\sum_{n=0}^\infty \left|\sum_{j=0}^n (j+1)(j+2)a_{j+2}b_{n-j}\right|r^n|dz|\nonumber\\[2mm]
&&+r^3  \int_{[z_1,z_2]}\sum_{n=0}^\infty \left|\sum_{j=0}^n (j+1)(n-j+1)a_{j+1}b_{n-j+1} \right|r^n|dz|\nonumber\\[2mm]
&&+r^3  \int_{[z_1,z_2]}\sum_{n=0}^\infty  \left|\sum_{j=0}^n (j+1)(n-j+1)a_{j+1}b_{n-j+1} \right|r^n|dz|\nonumber\\[2mm]
&& +r^3  \int_{[z_1,z_2]}\sum_{n=0}^\infty \left|\sum_{j=0}^n (n-j+1)(n-j+2)a_j b_{n-j+2}\right| r^n|dz|\nonumber\\[2mm]
&\leq & 2|z_2-z_1|(\Lambda_2+\Lambda_3)\frac{r^3}{1-r}.
\eea
Now, in view of \textrm{Lemma} $\ref{l4}$, we have
\beas I_5 &=& \left|\int_{[z_1,z_2]} (K_z(z)-K_z(0))dz-(K_{\ol{z}}(z)-K_{\ol{z}}(0))d\ol{z}\right|\nonumber\\[2mm]
&=& \left|\int_{[z_1,z_2]} \left(\sum_{n=2}^\infty nc_nz^{n-1}\right)dz-\left(\ol{\sum_{n=2}^\infty nd_nz^{n-1}}\right)d\ol{z}\right|\nonumber\eeas
\beas&\leq & \int_{[z_1,z_2]} \left| \sum_{n=2}^\infty n(c_n+d_n)r^{n-1}\right||dz|\\[2mm]
&\leq & |z_2-z_1|  \frac{\Lambda_1^2-1}{\Lambda_1}\sum_{n=2}^\infty r^{n-1}=  |z_2-z_1|  \frac{\Lambda_1^2-1}{\Lambda_1}\frac{r}{1-r}\quad\text{and}\\[2mm]
 I_6 &= & \left|\int_{[z_1,z_2]} zK_{zz}(z)dz-\ol{z}K_{\bar{z}\bar{z}}(z)d{\ol{z}}\right|\\[2mm]
&= & \left|\int_{[z_1,z_2]}z\sum_{n=2}^\infty n(n-1)c_nz^{n-2}dz-\ol{z}\ol{\sum_{n=2}^\infty n(n-1)d_nz^{n-2}}d{\ol{z}}\right|\\[2mm]
&\leq & r \int_{[z_1,z_2]}\left|\sum_{n=2}^\infty n(n-1)\left(c_n+d_n\right)r^{n-2}\right||dz|\\[2mm]
 &\leq & r |z_2-z_1| \frac{\Lambda_1^2-1}{\Lambda_1}\sum_{n=2}^\infty (n-1)r^{n-2}= |z_2-z_1| \frac{\Lambda_1^2-1}{\Lambda_1}\frac{r}{(1-r)^2}.
\eeas
Thus, using the inequalities for $I_1,I_2,\dots,I_6$ in (\ref{i7}), we have
\beas  |H(z_2)-H(z_1)|\geq |z_2-z_1|g_3(r),\eeas
where
\beas g_3(r)=1-3(\Lambda_2+\Lambda_3 )r^3-2(\Lambda_2+\Lambda_3)\frac{r^3}{1-r}-  \frac{\Lambda_1^2-1}{\Lambda_1}\frac{r(2-r)}{(1-r)^2},\eeas
which implies that
\beas g_3'(r)=-9r^2 (\Lambda_2+\Lambda_3 )-2(\Lambda_2+\Lambda_3)\frac{r^2(3-2r)}{(1-r)^2}-\frac{\Lambda_1^2-1}{\Lambda_1}\frac{2}{(1-r)^3} < 0\eeas
for $r\in (0,1)$. Hence, $g_3(r)$ is a strictly decreasing function of $r$ on $[0,1]$. Also,
$\lim_{r\to 0^+}g_3(r)=1$ and $\lim_{r\to 1^-}g_3(r)=-\infty$.
Consequently, there exists a unique $r_2\in(0,1)$ such that $g_3(r_2)=0$. Therefore, $H$ is univalent in $\D_{r_2}$.\\[2mm]
\indent To prove the second part of the theorem, let $z=r_2e^{i\theta}\in\pa\D_{r_2}$. In view of  \textrm{Lemma} $\ref{l4}$, (\ref{L10}) and (\ref{L11}), we have
\beas |H(z)| &=& \left|z|z|^2L_z(z)-\ol{z}|z|^2L_{\ol{z}}(z)+zK_z(z)-\ol{z}K_{\ol{z}}(z) \right|\\[2mm]
&\geq & r_2|K_z(0)- K_{\ol{z}}(0)|-r_2\left|(K_z(z)-K_z(0))-(K_{\ol{z}}(z)-K_{\ol{z}}(0))\right|\\[2mm]
&&-r_2^3\left|L_z(z)-L_{\ol{z}}(z)\right|\\[2mm]
&\geq & r_2-r_2\left| \sum_{n=2}^\infty n(c_n+d_n)r_2^{n-1}\right|-r_2^3\left|L_z(z)-L_{\ol{z}}(z)\right|\\[2mm]
&\geq & r_2- \frac{\Lambda_1^2-1}{\Lambda_1}\frac{r_2^2}{1-r_2}-(\Lambda_2+\Lambda_3)r_2^4=\rho_2.
\eeas
This completes the proof.
\end{proof}
Note that \textrm{Theorem} \ref{th4} deals with the case $\Lambda_1>1$. Therefore, we deal with the case $\Lambda_1=1$ in the following result.
\begin{theo}\label{th41}
Let $F(z)=r^2L(z)+K(z)$ be in $L_{Lh}(\D)$, where $L(z)$ is logharmonic and $K(z)$ is harmonic in $\D$ such that $F(0)=\lambda_F(0)-1=0$ with $\Lambda_K\leq 1, \Lambda_{L_z} \leq \Lambda_2 $ and $\Lambda_{L_{\ol{z}}} \leq \Lambda_3$. Then $D(F)$ is univalent in $|z|<r_2'$, where $r_2'\in (0,1)$ satisfies the equation
\beas 1-3(\Lambda_2+\Lambda_3 )r^3 -2(\Lambda_2+\Lambda_3)\frac{r^3}{1-r}=0.\eeas
Also, $D(F)$ contains a schlicht disk $\D_{\rho_2'}$ with $\rho_2'= r_2'- (\Lambda_2+\Lambda_3)r_2'^4$. 
\end{theo}
\begin{proof}
In view of \textrm{Lemma} $\ref{l4}$, we have $K(z)=a_1z+\ol{b_1z}$ with $||a_1|-|b_1||=1$. Using arguments similar to those used in the proof of \textrm{Theorem} \ref{th4}, we have
\bea\label{R4} |H(z_2)-H(z_1)|=\left|\int_{[z_1,z_2]}H_{z}(z)dz+H_{\ol{z}}(z)d\ol{z} \right|\geq  I_1-I_2-I_3-I_4,\eea
where
\bea\label{R5} I_1 &=& \left|\int_{[z_1,z_2]} K_z(z)dz-K_{\ol{z}}(z)d\ol{z} \right|\geq ||a_1|-|b_1|||z_2-z_1|=|z_2-z_1|,\\[2mm]
I_2&=&\left|\int_{[z_1,z_2]}2z\ol{z}L_z(z)dz-2\ol{z}zL_{\ol{z}}(z)d{\ol{z}}\right|,\;
I_3 = \left|\int_{[z_1,z_2]} -{\ol{z}}^2L_{\ol{z}}(z)dz+z^2L_z(z)d{\ol{z}}\right|,\nonumber\\[2mm]
I_4&=& \left|\int_{[z_1,z_2]} z^2\ol{z}L_{zz}(z)dz+z^2\ol{z}L_{\ol{z}z}(z)d{\ol{z}} -z{\ol{z}}^2L_{z\ol{z}}(z)dz
-z{\ol{z}}^2L_{{\bar{z}}{\bar{z}}}(z)d{\ol{z}}\right|.\nonumber
\eea
Using the inequalities (\ref{R1}), (\ref{R2}), (\ref{R3}), (\ref{R4}) and (\ref{R5}), we have $|H(z_2)-H(z_1)|\geq |z_2-z_1|g_4(r)$, where
\beas g_4(r)=1-3(\Lambda_2+\Lambda_3 )r^3 -2(\Lambda_2+\Lambda_3)\frac{r^3}{1-r}.\eeas
Therefore, $H$ is univalent in $\D_{r_2'}$, where $r_2'$ satisfies $g_4(r_2')=0$.\\[2mm]
\indent Now we show that $H(z)$ contains a schlicht disk $\D_{\rho_2'}$. Using (\ref{L10}) and (\ref{L11}), for $z=r_2'e^{i\theta}\in\pa\D_{r_2'}$, we have
\beas |H(z)| &=& \left|z|z|^2L_z(z)-\ol{z}|z|^2L_{\ol{z}}(z)+zK_z(z)-\ol{z}K_{\ol{z}}(z) \right|\\[2mm]
&\geq & r_2'|K_z(z)- K_{\ol{z}}(z)|-r_2'^4(\Lambda_2+\Lambda_3)\\[2mm]
&\geq & r_2'||a_1|-|b_1||- r_2'^4(\Lambda_2+\Lambda_3)\\[2mm]
&=& r_2'- (\Lambda_2+\Lambda_3)r_2'^4=\rho_2'.
\eeas
This completes the proof.
\end{proof}
\section*{Declarations}
\noindent{\bf Acknowledgement:} The first Author is supported by University Grants Commission (IN) fellowship (NO. NBCFDC/CSIR-UGC-DECEMBER-2023).\\[2mm]
{\bf Conflict of Interest:} The authors declare that there are no conflicts of interest regarding the publication of this paper.\\[2mm]
{\bf Availability of data and materials:} Not applicable


\begin{thebibliography}{33}
\bibitem{A2009} {\sc Z. Abdulhadi}, On the univalence of functions with logharmonic Laplacian, {\it Appl. Math. Comput.}, {\bf 215} (2009), 1900-1907.
\bibitem{AAA2012} {\sc Z. Abdulhadi, Y. Abu muhanna and R.M. Ali}, Landau's theorem for functions with logharmonic Laplacian, {\it Appl. Math. Comput.}, {\bf 218} (2012), 6798-6802.
\bibitem{AB1988} {\sc Z. Abdulhadi and D. Bshouty}, Univalent functions in $H\cdot\ol{H}(\D)$, {\it Trans. Amer. Math. Soc.}, {\bf 305}(2) (1988), 841-849.
\bibitem{AH1987} {\sc Z. Abdulhadi and W. Hengartner}, Spirallike logharmonic mappings, {\it Complex Variables Theory Appl.}, {\bf 9}(2-3) (1987), 121-130.
\bibitem{BL2019} {\sc X.X. Bai and M.S. Liu}, Landau-type theorems of polyharmonic mappings and log-$p$-harmonic mappings, {\it Complex Anal. Oper. Theory}, {\bf 13} (2019), 321-340.
\bibitem{CGH2000} {\sc H.H. Chen, P.M. Gauthier and W. Hengartner}, Bloch constants for planar harmonic mappings, {\it Amer. Math. Soc.}, {\bf 128}(11) (2000), 3231-3240.
\bibitem{CPW2011} {\sc S. Chen, S. Ponnusamy and X. Wang}, Bloch constant and Landau's theorem for planar $p$-harmonic mappings, {\it J. Math. Anal. Appl.}, {\bf 373} (2011), 102-110.
\bibitem{CPW2013} {\sc S. Chen, S. Ponnusamy and X. Wang}, On some properties of solutions of the $p$-harmonic equation, {\it Filomat}, {\bf 27}(4) (2013), 577-591.
\bibitem{CS1984} {\sc J. Clunie and T. Sheil-Small}, Harmonic univalent functions, {\it Annales Acad. Sci. Fenn. Series A. Mathematica}, {\bf 9} (1984), 3-25.
\bibitem{D2004} {\sc P. Duren}, Harmonic Mappings in the Plane, {\it Cambridge University Press}, 2004.
\bibitem{GK2003} {\sc I. Graham and G. Kohr}, Geometric Function Theory in One and Higher Dimensions, {\it Marcel Dekker Inc, New York}, 2003.
\bibitem{L1926} {\sc E. Landau}, Der Picard-Schottysche Satz und die Blochsche Konstanten, {\it Sitzungsber Press Akad, Wiss. Berlin Phys.-Math., Kl}, (1926) 467-474.
\bibitem{L2008} {\sc M.S. Liu}, Landau's theorems for biharmonic mappings, {\it Complex Var. Elliptic Equ.}, {\bf 53}(9) (2008), 843-855.
\bibitem{LL2014} {\sc  M.S. Liu and Z.X. Liu}, Landau-type theorems for $p$-harmonic mappings or log-$p$-harmonic mappings, {\it Applicable Analysis}, {\bf 93}(11) (2014), 2462-2477.
\bibitem{LL2019} {\sc M.S. Liu and  L.F. Luo}, Landau-type theorems for certain bounded biharmonic mappings, {\it Results Math.}, {\bf 74} (2019), 170.
\bibitem{LL2021} {\sc M.S. Liu and L.F. Luo}, Precise values of the Bloch constants of certain log-$p$-harmonic mappings, {\it Acta Math. Sci.}, {\bf 41}(1) (2021), 297-310.
\bibitem{M1980} {\sc P.T. Mocanu}, Starlikeness and convexity for nonanalytic functions in the unit disk, {\it Mathematica (Cluj)}, {\bf 22}(1) (1980), 77-83.
\bibitem{WPL2024} {\sc X. Wang, S. Ponnusamy and M.S. Liu}, On the univalence of certain polyharmonic mappings with bounded length distortions, {\it Acta Math. Sci.}, {\bf 44} (2024), 2125-2138.
\end{thebibliography}
\end{document}